\documentclass[a4paper,12pt]{article}
\usepackage{amsmath,amssymb}
\usepackage{bm}

\bmdefine{\xxx}{x}
\bmdefine{\aaa}{a}
\bmdefine{\bbb}{b}
\bmdefine{\eee}{e}
\bmdefine{\vvv}{v}
\bmdefine{\www}{w}
\bmdefine{\mmm}{m}
\bmdefine{\zerovec}{0}

\DeclareMathAlphabet{\mathscr}{U}{rsfs}{m}{n}

\newcommand{\CCC}{\mathbb{C}}

\newcommand{\RRR}{\mathbb{R}}

\newcommand{\FFF}{\mathbb{F}}

\newcommand{\NNN}{\mathbb{N}}

\newcommand{\ZZZ}{\mathbb{Z}}

\newcommand{\RRRplus}{\RRR_{\geq0}}
\newcommand{\RRRplusb}{{\RRR_{\geq0}}}

\newcommand{\msOOO}{\mathscr{O}}

\newcommand{\define}{\mathrel{:=}}

\newcommand{\rank}{\mathrm{rank}}

\newcommand{\transpose}{^\top}

\newcommand{\image}{{\mathrm{Im}}}

\newcommand{\grank}{{\mathrm{grank}}}
\newcommand{\trank}{{\mathrm{trank}}}
\newcommand{\nnrank}{{\mathrm{rank}_{\RRRplus}}}

\newcommand{\ranks}{{\rank_S}}
\newcommand{\rankr}{{\rank_\RRR}}
\newcommand{\rankc}{{\rank_\CCC}}

\newtheorem{thm}{Theorem}[section]

\newtheorem{example}[thm]{Example}
\newtheorem{lemma}[thm]{Lemma}

\newtheorem{remark}[thm]{Remark}

\makeatletter
\newcommand{\mysloppy}{\tolerance 9999 \hfuzz .5\p@ \vfuzz .5\p@}
\makeatother
\title{%
Maximal and typical nonnegative ranks of nonnegative tensors}
\author{%
Toshio SUMI\footnote{Faculty of Arts and Science, Kyushu University, Fukuoka, Japan},
Mitsuhiro MIYAZAKI%
\footnote{Department of  Mathematics, Kyoto University of Education, Kyoto, Japan}
and
Toshio SAKATA\footnote{Emeritus professor, Kyushu University, Fukuoka, Japan}}
\date{Version of \today}
\date{}

\begin{document}
\sloppy

\maketitle

\begin{abstract}
Let $N_1$, \ldots, $N_d$ be positive integers with 
$N_1\leq\cdots\leq N_d$.
Set $N=N_1\cdots N_{d-1}$.
We show in this paper that an integer $r$ is a typical nonnegative rank
of nonnegative tensors of format $N_1\times\cdots\times N_d$ if and only if
$r\leq N$ and $r$ is greater than or equals to the generic rank of
tensors over $\CCC$ of format $N_1\times\cdots\times N_d$.
We also show that the maximal nonnegative rank of nonnegative tensors
of format $N_1\times\cdots\times N_d$ is $N$.
\end{abstract}


\section{Introduction}

Recently high-dimensional array data are intensively studied in connection with data analysis.
In these areas of research, high-dimensional data are called tensors.
More precisely, a $d$-dimensional datum aligned in an $N_1\times\cdots\times N_d$ 
high-dimensional box form is called an $N_1\times \cdots\times N_d$ tensor
or a $d$-tensor of format $N_1\times\cdots\times N_d$.

Let $S$ be a set.
The set of $N_1\times\cdots\times N_d$ tensors with entries in $S$ is denoted
$S^{N_1\times\cdots\times N_d}$.
An element $T\in S^{N_1\times\cdots\times N_d}$ is denoted by
$T=(t_{i_1\cdots i_d})_{1\leq j\leq N_j}$ or simply
$T=(t_{i_1\cdots i_d})$.
The case where $d=2$, $S^{N_1\times N_2}$ is nothing but the set of $N_1\times N_2$
matrices with entries in $S$.
If $S$ is a field, the rank of an element $M\in S^{N_1\times N_2}$ is defined in linear algebra.
The rank of $M$ is identical with the minimum number $r$ such that $M$ can be expressed
as a sum of $r$ rank 1 matrices, where we set empty sum to be zero.
Further, a nonzero matrix is rank 1 if and only if it can be expressed as a product
of a nonzero column vector and a nonzero row vector.
We generalize the notion of rank to the general case by using this 
characterization of ranks of matrices.

Suppose $S$ is a subset of a commutative ring with identity which is closed 
under addition and multiplication and contains 0 and 1.
A nonzero element $T=(t_{i_1\cdots i_d})\in S^{N_1\times\cdots\times N_d}$ is
defined to be $S$-rank 1 if there are 
$\vvv_j=(v_{j1}, \ldots, v_{jN_j})\transpose\in S^{N_j}$ for $1\leq j\leq d$
such that 
$
t_{i_1\cdots i_d}=\prod_{j=1}^d v_{ji_j}
$
for any $(i_1,\ldots, i_d)$.
For a general element $T\in S^{N_1\times\cdots\times N_d}$,
we define the $S$-rank of $T$ as the minimum number $r$ such that $T$ can be
expressed as a sum of $r$ $S$-rank 1 tensors.
The $S$-rank of $T$ is denoted $\ranks T$.
We attach the prefix ``$S$-'' since the rank depends on $S$.
In fact, there is an element of $\RRR^{2\times 2\times 2}$ such that
$\rankr T=3$ and $\rankc T=2$.

In this paper, we consider the case where $S=\RRRplus\define\{x\in\RRR\mid x\geq0\}$.
Since it is important to work in the set of nonnegative real numbers in the
statistical point of view, $\RRRplus$-rank is studied by several researchers
\cite{all etal,qcl}.
They called an element $\RRRplus^{N_1\times\cdots\times N_d}$ a nonnegative tensor
and called the $\RRRplus$-rank nonnegative rank.
Allman et al. \cite{all etal} showed that under mild condition, nonnegative rank 
of a nonnegative tensor is identical
with its $\RRR$-rank if its $\RRR$-rank is less than or equals to 2.

We consider in this paper typical nonnegative ranks of $\RRRplus^{N_1\times \cdots\times N_d}$
(see section \ref{sec:pre} for definition).
Suppose that $N_1\leq \cdots\leq N_d$ and set $N=N_1\cdots N_{d-1}$.
We show 
that $r$ is a typical nonnegative rank of nonnegative 
$N_1\times\cdots\times N_d$
tensors if and only if $r\leq N$ and $r$ is greater than or equals to
the minimal typical $\RRR$-rank.
In the way of proving this fact, we also show that $N$ is the maximal nonnegative
rank of $N_1\times\cdots\times N_d$ nonnegative tensors.


\section{Preliminaries}
\label{sec:pre}

In this section, we fix notation, state definitions and recall some basic facts
in order to prepare for the main argument.
Let $\FFF$ be the real number field $\RRR$ or the complex number field $\CCC$
and let $n$ be a positive integer.
We denote by $\FFF^n$ the vector space of $n$-dimensional column vectors with entries in $\FFF$
and by $\eee_i$ the element of $\FFF^n$ whose $i$-th entry is 1 and other entries are 0.
By abuse of notation, we use the same symbol $\eee_i$ to express an element of
$\FFF^n$ and $\FFF^m$ even if  $n\neq m$.
The meaning will be clear from the context.

Let $N_1, \ldots, N_d$ be positive integers.
Since an element $\FFF^{N_1}\otimes_\FFF\cdots\otimes_\FFF\FFF^{N_d}$ is expressed 
uniquely as
$$
\sum_{i_1=1}^{N_1}\cdots
\sum_{i_d=1}^{N_d}
t_{i_1\cdots i_d}\eee_{i_1}\otimes\cdots\otimes\eee_{i_d},
$$
we identify $\FFF^{N_1\times\cdots\times N_d}$ and 
$\FFF^{N_1}\otimes_\FFF\cdots\otimes_\FFF\FFF^{N_d}$
by the correspondence
$$
(t_{i_1\cdots i_d})\leftrightarrow
\sum_{i_1=1}^{N_1}\cdots
\sum_{i_d=1}^{N_d}
t_{i_1\cdots i_d}\eee_{i_1}\otimes\cdots\otimes\eee_{i_d}.
$$
We introduce the Euclidean topology to $\FFF^{N_1\times\cdots\times N_d}$
by identifying 
$\FFF^{N_1\times\cdots\times N_d}$
with
$\FFF^{N_1\cdots N_d}$.

Let $S$ be $\RRRplus\define\{x\in\RRR\mid x>0\}$, $\RRR$ or $\CCC$.
We interpret $\FFF=\RRR$ if 
$S=\RRRplus$ or $S=\RRR$ 
and $\FFF=\CCC$ if $S=\CCC$
in the rest of this section.
We define
$$
\Phi_1^S\colon S^{N_1}\times\cdots\times S^{N_d}\to S^{N_1\times\cdots\times N_d}
$$
by
$\Phi_1^S((v_{11},\ldots,v_{1N_1})\transpose,\ldots,
(v_{d1},\ldots,v_{dN_d})\transpose)=(v_{1i_1}\cdots v_{di_d})_{1\leq i_j\leq N_j}$.
Further, we define for $r\geq 2$, 
$\Phi_r^S\colon (S^{N_1}\times\cdots\times S^{N_d})^r\to S^{N_1\times\cdots\times N_d}$
by
$\Phi_r^S(a_1,\ldots, a_r)=\Phi_1^S(a_1)+\cdots+\Phi_1^S(a_r)$.
We define the $S$-rank of a nonzero element $T$ of $S^{N_1\times\cdots\times N_d}$
is $r$ if $T\in\image \Phi_r^S\setminus\image\Phi_{r-1}^S$,
where we interpret $\image\Phi_0^S$ to be the set consists of zero element only,
and we define the $S$-rank of the zero element of $S^{N_1\times\cdots\times N_d}$ to be 0.
If the set of elements $S^{N_1\times\cdots\times N_d}$ whose $S$-rank is $r$ contains
a nonempty Euclidean open subset of $\FFF^{N_1\times\cdots\times N_d}$, we say that 
$r$ is a typical $S$-rank of $N_1\times\cdots\times N_d$ tensors.

It is known that there exists a unique typical $\CCC$-rank of $N_1\times\cdots\times N_d$
tensors over $\CCC$ which is called the generic rank of $N_1\times\cdots\times N_d$
tensors over $\CCC$ and is 
denoted $\grank_\CCC(N_1,\ldots,N_d)$.
Further, it is known that the generic rank of $N_1\times\cdots\times N_d$ tensors over
$\CCC$ is identical with the minimal typical $\RRR$-rank of $N_1\times\cdots\times N_d$
tensors.
See \cite[Theorem 2]{bt} or \cite[Chapter 6]{book}.

In this paper, we consider the case where $S=\RRRplus$.
In this case an element of $S^{N_1\times\cdots\times N_d}$ is called a
nonnegative tensor of format $N_1\times\cdots\times N_d$
or an $N_1\times\cdots\times N_d$ nonnegative tensor, an $S$-rank
is called a nonnegative rank and a typical $S$-rank is called a typical
nonnegative rank.
In the case where $d=2$, i.e. the matrix case is studied extensively as a
nonnegative matrix factorization.

\begin{example}\rm
The matrix
$
\begin{pmatrix}1&0&1&0\\
0&1&0&1\\
1&0&0&1\\
0&1&1&0
\end{pmatrix}
$
has $\RRR$-rank 3 and nonnegative rank 4.
In fact, the column vectors of this matrix span the convex polyhedral
cone
$\{(x_1,x_2,x_3,x_4)\transpose\in\RRRplus\mid x_1+x_2-x_3-x_4=0\}$.
See \cite[Example 3.6]{sta}.
\end{example}


\section{Main result}

In this section, we study maximal and typical nonnegative ranks of $N_1\times\cdots\times N_d$
nonnegative tensors with $N_1\leq\cdots\leq N_d$.
Set $N=N_1\cdots N_{d-1}$.
First we state the following upper bound of nonnegative ranks of nonnegative tensors
of format $N_1\times\cdots\times N_d$.

\begin{lemma}
\label{lem:u bound}
Let $T$ be a nonnegative tensor of format $N_1\times\cdots\times N_d$.
Then $\nnrank T\leq N$.
\end{lemma}
\begin{proof}
Set $T=(t_{i_1\cdots i_d})$.
Then
$$
T=\sum_{i_1=1}^{N_1}\cdots\sum_{i_{d-1}=1}^{N_{d-1}}\eee_{i_1}\otimes\cdots\otimes\eee_{i_{d-1}}
\otimes(t_{i_1\cdots i_{d-1}1}, \ldots, t_{i_1\cdots i_{d-1} N_d})\transpose.
$$
Thus, $\nnrank T\leq N$.
\end{proof}

Next we show the following.

\begin{lemma}
\label{lem:max trank}
$N$ is a typical nonnegative rank of nonnegative tensors of format
$N_1\times\cdots\times N_d$.
\end{lemma}
\begin{proof}
Set $I=\{(i_1,\ldots, i_d)\mid 1\leq i_j\leq N_j$ for $1\leq j\leq d$
and $i_1+\cdots+i_d\equiv 0\pmod {N_d}\}$.
Then $\#I=N$.
Set also
$
T_0=\sum_{(i_1,\ldots,i_d)\in I}\eee_{i_1}\otimes\cdots\otimes \eee_{i_d}
$
and
$\epsilon=\frac{1}{3N}$.
We show that if $T\in\RRRplus^{N_1\times\cdots\times N_d}$ and
$||T-T_0||<\epsilon$, then $\nnrank T=N$,
where $||T-T_0||$ denotes the Euclidean norm.

Assume the contrary and set $r=\nnrank T$.
Then by Lemma \ref{lem:u bound}, we see that $r<N$.
Further set 
$
T=\sum_{k=1}^r \aaa_{1k}\otimes\cdots\otimes\aaa_{dk}
$,
where $\aaa_{jk}\in\RRRplus^{N_j}$ for any $1\leq j\leq d$ and $1\leq k\leq r$.
Set
$\aaa_{jk}=(a_{jk1}, \ldots, a_{jkN_j})\transpose$ and for
$(i_1,\ldots, i_d)\in I$, we set 
$$
K_{(i_1,\ldots, i_d)}\define\{k\mid 1\leq k\leq r,\ a_{jkl}<a_{jk i_j}
\mbox{ for any $1\leq j\leq d$ and $l\neq i_j$.}\}
$$
It is clear that if
$(i_1,\ldots, i_d)\neq(\imath'_1,\ldots,\imath'_d)$, then 
$K_{(i_1,\ldots, i_d)}\cap K_{(\imath'_1,\ldots,\imath'_d)}=\emptyset$.

Suppose that $(i_1,\ldots, i_d)\in I$ and $k_0\not\in K_{(i_1,\ldots,i_d)}$.
Then there are $j$ and $l$ with $i_j\neq l$ and $a_{jk_0 l}\geq a_{jk_0i_j}$.
Since $(i_1,\ldots, i_{j-1},l,i_{j+1},\ldots,i_d)\not\in I$ 
by the definition of $I$, we see that
$(i_1,\ldots,i_{j-1}, l,i_{j+1},\ldots i_d)$ entry of $T_0$ is 0.
Thus, $(i_1,\ldots,i_{j-1}, l,i_{j+1},\ldots i_d)$ entry of $T$ is 
less than $\epsilon$.
Therefore,
$$
a_{1k_0i_1}\cdots a_{jk_0i_j}\cdots a_{dk_0i_d}\leq
a_{1k_0i_1}\cdots a_{jk_0 l}\cdots a_{dk_0i_d}<\epsilon,
$$
since $a_{jki}\geq 0$ for any $j$, $k$ and $i$.

Set
$T'=\sum_{k\not\in K_{(i_1,\ldots, i_d)}}\aaa_{1k}\otimes\cdots\otimes \aaa_{dk}$.
Then, by the above argument, we see that the $(i_1,\ldots, i_d)$ entry of $T'$ is less than
$\epsilon(r-\#K_{(i_1,\ldots, i_d)})<\epsilon N$.
On the other hand, since $(i_1, \ldots, i_d)$ entry of $T_0$ is 1,
and $||T-T_0||<\epsilon$, we see that $(i_1, \ldots, i_d)$ entry of $T$
is greater than $1-\epsilon$.
Since $\epsilon=\frac{1}{3N}$, we see that $1-\epsilon>\epsilon N$ 
and therefore, $T\neq T'$, i.e.
$K_{(i_1,\ldots, i_d)}\neq\emptyset$.
Thus, we see that
$N=\#I\leq r$
since 
$K_{(i_1,\ldots, i_d)}\cap K_{(\imath'_1,\ldots,\imath'_d)}=\emptyset$
if $(i_1,\ldots, i_d)\neq(\imath'_1,\ldots,\imath'_d)$. 
This contradicts to the assumption that $r<N$.
\end{proof}
By Lemmas \ref{lem:u bound} and \ref{lem:max trank}, we see the
following fact.

\begin{thm}
\label{thm:max nn rank}
The maximal nonnegative rank of nonnegative tensors of format
$N_1\times\cdots\times N_d$ is $N$.
\end{thm}

Next we consider the set of typical nonnegative ranks of nonnegative
tensors of format $N_1\times\cdots\times N_d$.
First we recall the following fact.

\begin{lemma}[{see e.g. \cite[Chapter 6]{book}}]
\label{lem:g than grank}
An integer $r$ is greater than or equals to 
$\grank_\CCC(N_1,\ldots,N_d)$ if and only if  the Jacobian
of $\Phi_r^\CCC$ is row full rank for some 
$(a_1,\ldots,a_r)\in(\CCC^{N_1}\times\cdots\times\CCC^{N_d})^r$.
In particular, if $r\geq\grank_\CCC(N_1,\ldots,N_d)$,
then there is a nonzero maximal minor of Jacobian of $\Phi_r^\CCC$
as a polynomial.
\end{lemma}
Note that every Jacobian of $\Phi_r^\CCC$ as a polynomial is 
a polynomial with coefficients in $\ZZZ$.
We also note the following basic fact.

\begin{lemma}
\label{lem:inf domain}
Let $R$ be an integral domain, $X_1$, \ldots, $X_n$ indeterminates and
$A_1$, \ldots, $A_n$ infinite subsets of $R$.
Then for any nonzero polynomial $f\in R[X_1,\ldots, X_n]$,
there are $a_1$, \ldots, $a_n$ with $a_j\in A_j$ for $1\leq j\leq n$
such that $f(a_1,\ldots, a_n)\neq 0$.
\end{lemma}
The following fact is easily verified.

\begin{lemma}
\label{lem:sum rank}
Let $T$ be a nonnegative tensor with nonnegative rank $s$.
Suppose
$
T=T_1+\cdots+T_s
$
with $\nnrank T_j=1$ for $1\leq j\leq s$, then
$$
\nnrank (T_1+\cdots+T_r)=r
$$
for any $1\leq r\leq s$.
\end{lemma}
We also note the following fact.

\begin{lemma}
\label{lem:g than min}
Let $r$ be a typical nonnegative rank of nonnegative tensors
of format $N_1\times\cdots\times N_d$.
Then $r\geq\min(\trank_\RRR(N_1,\ldots,N_d))$.
\end{lemma}
\begin{proof}
By the definition of typical nonnegative rank,
there is a nonempty Euclidean open subset $U$ of $\RRRplus^{N_1\times\cdots\times N_d}$
such that $\nnrank T=r$ for any $T\in U$.
Since $\rankr T\leq\nnrank T$, $U$ consists of tensors whose $\RRR$-rank
is less than or equals to $r$.
Therefore, $r\geq\min(\trank_\RRR(N_1,\ldots,N_d))$.
\end{proof}

Now we state the following.

\begin{thm}
\label{thm:t rank range}
The set of typical nonnegative ranks of nonnegative tensors
of format $N_1\times\cdots\times N_d$ is
$
\{r\mid r\in\ZZZ, \grank_\CCC(N_1,\ldots, N_d)\leq r\leq N\}$.
\end{thm}
\begin{proof}
Let $r$ be a typical nonnegative rank of nonnegative tensors of format
$N_1\times\cdots\times N_d$.
By Lemma \ref{lem:u bound}, we see that $r\leq N$.
Further, by Lemma \ref{lem:g than min},
$r\geq\min(\trank_\RRR(N_1,\ldots, N_d))=\grank_\CCC(N_1,\ldots, N_d)$.

Conversely assume that
$\grank_\CCC(N_1,\ldots, N_d)\leq r \leq N$.
By Lemma \ref{lem:g than grank}, we see that there is a maximal minor
$f$ of the Jacobian of $\Phi_r^\CCC$ which is nonzero as a polynomial.
Since $N$ is a typical nonnegative rank of nonnegative tensors
of format $N_1\times\cdots\times N_d$, we see that there is a Euclidean open
subset $\msOOO$ of $\RRRplus^{N_1\times\cdots\times N_d}$ such that
$\nnrank T=N$ for any $T\in\msOOO$.

Since 
$$
\Phi_N^\RRRplusb\colon(\RRRplus^{N_1}\times\cdots\times\RRRplus^{N_d})^N
\to
\RRRplus^{N_1\times\cdots\times N_d}
$$
is continuous and identical with
$$
(\RRRplus^{N_1}\times\cdots\times\RRRplus^{N_d})^r
\times
(\RRRplus^{N_1}\times\cdots\times\RRRplus^{N_d})^{N-r}
\to
\RRRplus^{N_1\times\cdots\times N_d},
$$
$(A_1,A_2)\mapsto \Phi_r^\RRRplusb(A_1)+\Phi_{N-r}^\RRRplusb(A_2)$,
we see that there are open subsets 
$U_1$ of $(\RRRplus^{N_1}\times\cdots\times\RRRplus^{N_d})^r$
and
$U_2$ of $(\RRRplus^{N_1}\times\cdots\times\RRRplus^{N_d})^{N-r}$
such that
$\Phi_r^\RRRplusb(U_1)+\Phi_{N-r}^\RRRplusb(U_2)\subset\msOOO$.

By Lemma \ref{lem:inf domain}, we see that there exists $A_1\in U_1$
such that $f(A_1)\neq 0$.
Therefore, by the inverse function theorem,
we see that $\image \Phi_r^\RRRplusb(U_1)$ contains an open
neighborhood $\msOOO'$ of $\Phi_r^\RRRplusb(A_1)$.
On the other hand, by Lemma \ref{lem:sum rank},
we see that every element of $\image \Phi_r^\RRRplusb(U_1)$
has nonnegative rank $r$.
Thus, we see that every element of $\msOOO'$ has nonnegative
rank $r$.
Therefore, $r$ is a typical nonnegative rank.
\end{proof}

\begin{remark}
\rm
The fact that the set of typical nonnegative ranks form a consecutive
integers can also be proved along the same line with
\cite[Theorem 2.2]{bbo}.
\end{remark}


\end{document}